\documentclass[11pt]{article}

\usepackage{amsmath,amsthm,mathrsfs}
\usepackage{amssymb,latexsym}

\usepackage{enumerate}

\usepackage[mathscr]{euscript}

\newcommand{\BB}{{\mathcal B}}

\newcommand{\EE}{{\mathcal E}}

\newcommand{\BD}{{\mathbb D}}

\newcommand{\BM}{{\mathbb M}}

\newcommand{\BR}{{\mathbb R}}

\newcommand{\BBr}{{\mathscr B}}

\newcommand{\fch}{{\mathbf{1}}}

\newtheorem{theorem}{\bf Theorem}[section]
\newtheorem{proposition}[theorem]{\bf Proposition}
\newtheorem{lemma}[theorem]{\bf Lemma}

\theoremstyle{definition}
\newtheorem{definition}[theorem]{Definition}

\newtheorem{remark}[theorem]{Remark}

\numberwithin{equation}{section}

\begin{document}

\title {On approximation of the Dirichlet problem for divergence form operator
by Robin problems\footnote{Supported by Polish National Science Centre (grant no. 2016/23/B/ST1/01543).}}
\author {Andrzej Rozkosz and Leszek S\l omi\'nski\\
{\small Faculty of Mathematics and Computer Science,
Nicolaus Copernicus University} \\
{\small  Chopina 12/18, 87--100 Toru\'n, Poland}}
\date{}
\maketitle
\begin{abstract}
We show that, under natural assumptions solutions, of  Dirichlet problems for uniformly elliptic divergence form operator can be approximated pointwise by solutions of some versions of Robin problems.
The proof is based on stochastic representation of solutions  and properties of reflected diffusions corresponding to divergence form operators.
\end{abstract}
{\small \noindent{\bf Keywords:}  Robin problem, Dirichlet problem, divergence form operator, stochastic representation, reflected diffusion.
}

\section{Introduction}
\label{sec1}

Let $D$ be a bounded Lipschitz domain in $\BR^d$, $d\ge3$, and
\[
L=\sum^d_{i,j=1}\partial_{x_i}(a_{ij}(x)\partial_{x_j})
\]
be
the operator with measurable coefficients $a_{ij}:D\rightarrow\BR$ such that
\begin{equation}
\label{eq1.1}
a_{ij}=a_{ji},\qquad \Lambda^{-1}|\xi|^2\le \sum^{d}_{i,j=1}a_{ij}(x)\xi_i\xi_j\le\Lambda|\xi|^2,\quad x\in D,\,\,\xi\in\BR^d,
\end{equation}
for some $\Lambda\ge1$. For $\lambda>0$, $f:D\rightarrow\BR$, $g:\partial D\rightarrow\BR$ and $n\ge1$
we consider the following  boundary-value problem
\begin{equation}
\label{eq1.2}
-Lu_n+\lambda u_n=f\quad\mbox{in }D,\qquad-(a\nabla u_n)\cdot\mathbf{n}
+nu_n=ng\quad \mbox{on }\partial D,
\end{equation}
where  $a=\{a_{ij}\}_{1\le i,j\le d}$ and $\mathbf{n}(x)$ is the inward unit normal at $x\in \partial D$. Note that (\ref{eq1.2}) is a particular version of Robin problem (also known as Fourier problem 
or the third boundary-value problem).
It is known (see, e.g., \cite[Appendix I, Section 4.4]{G}) that if $f\in L^2(D)$, $g\in H^1(D)$
and in the boundary condition in (\ref{eq1.2}) the trace of $g$ i used, then for  each $n\ge1$ there exists a unique weak solution of (\ref{eq1.2}) and $u_n\rightarrow u$ in $H^1(D)$ as $n\rightarrow\infty$, where $u$ is the unique weak solution of the Dirichlet problem
\begin{equation}
\label{eq1.3}
-Lu+\lambda u=f\quad\mbox{in }D,\qquad u=g\quad\mbox{on }\partial D.
\end{equation}

If $f\in L^p(D)$ with $p>d$ and $g\in H^1(D)\cap C(\partial D)$,
then $u_n,u$ have continuous versions  and one may ask whether
$u_n\rightarrow u$ for every $x\in\bar D$. In this note, we give positive answer to this question.
Our proof is quite simple and is based on stochastic representation of solutions of (\ref{eq1.2}), (\ref{eq1.3}). But let us stress that in the proof of our convergence results we use  deep results from \cite{FT1, FT2} (see also \cite{BH} for the case $L=(1/2)\Delta$) saying that one can construct a reflected diffusion $\BM$ on $\bar D$  associated with $L$ having a strong Feller resolvent.

\section{Preliminaries}

In this paper,  $D\subset \BR^d$, $d\ge3$, is a bounded Lipschitz domain (for a definition see, e.g., \cite[Exercise 5.2.2]{FOT}), $\bar D=D\cup\partial D$. We denote by $m$ or simply by $dx$  the $d-$dimensional Lebesgue measure, and by $\sigma$ the surface measure on $\partial D$. $\BB(\bar D)$ is the set of Borel subsets of $\bar D$, $\BB_b(\bar D)$ (resp. $C(\bar D)$) is the set of bounded Borel (resp. continuous) functions on $\bar D$.
To shorten notation, we write $L^2(D)$ instead of $L^2(D;m)$ and $L^2(\partial D)$ instead of $L^2(\partial D;\sigma)$.

We assume that  the matrix $a$ satisfies (\ref{eq1.1}) and  consider the Dirichlet form $(\EE,D(\EE))$ on $L^2(D)$ defined by
\begin{equation}
\label{eq2.2}
\EE(u,v)=\sum^d_{i,j=1}\int_Da_{ij}(x)\frac{\partial u}{\partial x_i}(x)
\frac{\partial v}{\partial x_j}(x)\,dx,\quad u,v\in D(\EE):=H^1(D),
\end{equation}
where $H^1(D)$ is the usual Sobolev space of order 1, and for $\lambda>0$ set $\EE_{\lambda}(u,v)=\EE(u,v)+\lambda(u,v)$, where $(\cdot,\cdot)$ is the usual inner product in $L^2(D;m)$. We denote by $(T_t)_{t>0}$ the strongly continuous semigroup of Markovian symmetric operators on $L^2(D)$ associated with $\EE$ (see \cite[Section 1.3]{FOT}).

In the paper, we define quasi-notions (exceptional sets, quasi-continuity) with respect to $(\EE,H^1(D))$. We will say that a property of points in $\bar D$ holds quasi everywhere (q.e. in abbreviation) if it holds outside some exceptional set. It is known (see \cite[Lemma 2.1.4, Theorem 2.1.3]{FOT}) that each element of $H^1(D)$ admits a quasi-continuous $m$-version, which we denote by $\tilde u$, and $\tilde u$ i q.e. unique for every $u\in H^1(D)$.

In \cite[Theorems 2.1 and 2.2]{FT2} (see also \cite{FT1})
it is proved that  there exists a conservative diffusion proces
$\BM=\{(X,P_x),x\in\bar D\}$ on $\bar D$ associated with the Dirichlet form (\ref{eq2.2})
in the sense that the transition density of $\BM$ defined as
\[
p_t(x,B)=P_x(X_t\in B),\quad t>0,\,x\in \bar D,\quad B\in\BBr(\bar D),
\]
has the property that
\[
P_tf\quad\mbox{is an $m$-version of}\quad T_tf\quad\mbox{for every }f\in\BB_b(\bar D),
\]
where $P_tf(x)=\int_Df(y)p_t(x,dy)=E_xf(X_t)$. Moreover,  $(P_t)_{t>0}$ is strongly Feller in the sense that $P_t(\BB_b(\bar D))\subset C(\bar D)$ and $\lim_{t\downarrow0}P_tf(x)=f(x)$ for $x\in\bar D$, $f\in C(\bar D)$. In particular (see \cite[Exercise 4.2.1]{FOT}), the transition density satisfies the following absolute continuity condition:
$p_t(x,\cdot)\ll m$ for any $t>0$, $x\in\bar D$.

We denote by $(R_{\alpha})_{\alpha>0}$ the resolvent associated with $\BM$ (or with $(P_t)_{t>0})$, that is
\[
R_{\alpha}f(x)=E_x\int^{\infty}_0e^{-\alpha t}f(X_t)\,dt, \quad f\in\BB_b(\bar D).
\]
Of course
\[
R_{\alpha}f(x)=\int_{\bar D}r_{\alpha}(x,y)f(y)\,dy,\quad\mbox{where}\quad  r_{\alpha}(x,y)=\int^{\infty}_0e^{-\alpha t}p_t(x,y)\,dt.
\]
For a Borel measure $\mu$ on $\bar D$ we also set
\[
R_{\alpha}\mu(x)=\int_{\bar D}r_{\alpha}(x,y)\,\mu(dy),\quad x\in\bar D,\quad \alpha>0,
\]
whenever the integral makes sense.

By \cite[Lemma 5.1, Theorem 5.1]{FT2},  the surface measure $\sigma$ belongs to the space of smooth measures in the strict sense, and hence, by \cite[Theorem 5.1.7]{FOT}, there is  a unique positive continuous additive functional of $\BM$ in the strict sense with Revuz measure $\sigma$. In what follows we denote it by $A$. For $g\in\BB_b(\bar D)$ let  $g\cdot\sigma$ be the measure on $\bar D$ defined by  $g\cdot\sigma(B)=\int_Bg(x)\,\sigma(dx)$, $B\in\BB(\bar D)$. Note that for any $g\in\BB_b(\bar D)$ we have
\[
R_{\alpha}(g\cdot\sigma)(x)=E_x\int^{\infty}_0e^{-\alpha t}g(X_s)\,dA_s,\quad x\in\bar D.
\]
Indeed, by \cite[Theorem 5.1.3]{FOT} the above equality holds for $m$-a.e. $x\in \bar D$, and hence, by \cite[Theorem A.2.17]{CF}, for every $x\in\bar D$ because $p_t$ satisfies the absolute continuity condition and for any nonnegative $g\in\BB_b(\bar D)$, both sides of the above equality are  $\alpha$-excessive functions.
Also note that the support of  $A$ is contained in $\partial D$. Hence
\begin{equation}
\label{eq2.4}
\int^t_0g(X_s)\,dA_s=\int^t_0\fch_{\partial D}(X_s)g(X_s)\,dA_s,\quad P_x\mbox{a.s.},\quad x\in\bar D
\end{equation}
(for more details see the beginning of the proof of Lemma \ref{lem4.1}).  It follows that in fact the right-hand side of (\ref{eq2.4}) is well defined for $g\in\BB_b(\partial D)$.

\begin{remark}
If, in addition, $\frac{\partial a_{ij}}{\partial x_i}\in L^{\infty}(D)$, $i,j=1,\dots,d$, then $X=(X^1,\dots,X^d)$ has the following Skorohod representation: for $i=1,\dots,d$ and every $x\in\bar D$
\begin{equation}
\label{eq2.8}
X^i_t=X_0^i+M^i_t+N^i_t,\quad t\ge0,\quad P_x\mbox{-a.s.},
\end{equation}
where $M^i$ are martingale additive functionals in the strict sense with covariations
\[
\langle M^i,M^j\rangle_t=2\int^t_0a_{ij}(X_s)\,ds, \quad t\ge0,\quad P_x\mbox{-a.s.},
\]
and
\[
N^i_t=\sum^d_{j=1}\int^t_0\frac{\partial a_{ij}}{\partial x_j}(X_s)\,ds
+\sum^d_{j=1}\int^t_0a_{ij}(X_s){\mathbf n}_j(X_s)\,dA_s,\quad t\ge0,\quad P_x\mbox{-a.s.}
\]
In case of the classical Dirichlet form defined by
\[
\BD(u,v)=\frac12\sum_{i=1}\int_D\frac{\partial u}{\partial x_i}(x)
\frac{\partial v}{\partial x_j}(x)\,dx,\quad u,v\in H^1(D),
\]
i.e. if $a=\frac12 I$, the process $\BM$ is called a reflecting Brownian motion. By L\'evy's characterization of Brownian motion, the representation (\ref{eq2.8}) reads
\begin{equation}
\label{eq2.9}
X^i_t-X_0^i=B^i_t+\frac12\int^t_0\mathbf{n}_i(X_s)\,dA_s,\quad t\ge0,\quad P_x\mbox{-a.s.},
\end{equation}
where $B=(B^1,\dots,B^d)$ is a standard Brownian motion. For the proof of (\ref{eq2.9}) see \cite[Example 5.2.2]{FOT} and for the general case (\ref{eq2.8}) see \cite[Theorem 2.3]{FT2}. In case $a$ is a general function satisfying (\ref{eq1.2}) some representation of $X$ (Lyon's--Zheng--Skorohod decomposition) is given in \cite{RS:SM} (for bounded $C^2$ domain $D$ and $x\in D$).
\end{remark}

Let
\[
\tau_D=\inf\{t>0:X_t\notin D\},\qquad X^D_t=\begin{cases} X_t, &t<\tau_D\\
\partial, &t\ge\tau_D,
\end{cases}
\]
where $\partial$ is a point adjoined to $D$ as an isolated point  (cemetery state). We adopt the convention that every function $f$ on $D$ is extended to $\bar D\cup\partial$ by setting $f(\partial)=0$.

We denote by  $\BM^{\lambda}$  the canonical subprocess
of $\BM$ with respect to the multiplicative functional $e^{-\lambda t}$.
For its detailed construction we refer to \cite[Section A.2]{FOT}. 
Here let us only note that we may assume that $\BM^{\lambda}=(X^{\lambda},P_x)$ is defined on the same probability space on which $\BM$ is defined and
\[
X^{\lambda}_t=
\begin{cases}
X_t, & t<Z/\lambda,\\
\partial, &t\ge Z/\lambda,
\end{cases}
\]
where $Z$ is a nonnegative random variable independent of $(X_t)_{t\ge0}$ having exponential distribution with mean 1.

\section{Weak and probabilistic solutions}

For the convenience of the reader, 
below we recall variational formulation of problems (\ref{eq1.2}), (\ref{eq1.3}). For more details and comments we refer to \cite[Appendix I]{G}.

\begin{definition}
\label{def3.1}
(i) Let $f\in L^2(D)$, $g\in L^2(\partial D)$. A function $u_n\in H^1(D)$ is called a weak solution of (\ref{eq1.2})
if for every  $v\in H^1(D)$,
\begin{equation}
\label{eq2.3}
\EE_{\lambda}(u_n,v)=\int_Dfv\,dx+n\int_{\partial D}(g-u_n)v\,d\sigma.
\end{equation}
(ii) Let $f\in L^2(D)$, $g\in H^1(D)$. A function $u\in H^1(D)$ is called a weak
solution of (\ref{eq1.3}) if $u-g\in H^1_0(D)$ and for every $v\in H^1_0(D)$,
\[
\EE_{\lambda}(u,v)= \int_Dfv\,dx.
\]
\end{definition}

The existence and uniqueness of weak solutions of (\ref{eq1.2}), (\ref{eq1.3}) is well known. For proofs by classical variational methods we refer for instance to \cite[Appendix I]{G}. In Proposition \ref{prop2.4} below we give proofs by using the probabilistic potential theory.  The advantage  of using these less classical methods lies in the fact that they  provide probabilistic representations of quasi-continuous versions of weak solutions. We would like to stress that the proof of Proposition \ref{prop2.4} is  simply a compilation of known facts. We provide it for completeness and later use.


\begin{proposition}
\label{prop2.4}
\begin{enumerate}[\rm(i)]
\item Let $f\in L^2(D)$, $g\in L^2(\partial D)$.  Then there exists a unique weak solution $u_n$ of \mbox{\rm(\ref{eq1.2})} and $\tilde u_n$ defined q.e. on $D$ by
\begin{equation}
\label{eq2.5}
\tilde u_n(x)=E_x\int^{\infty}_0e^{-\lambda t-nA_t} (f(X_t)\,dt+ng(X_t)\,dA_t)
\end{equation}
is a quasi-continuous $m$-version of $u_n$.

\item Let $f\in L^2(D)$, $g\in H^1(D)$. Then there exists a unique weak solution $u$ of \mbox{\rm(\ref{eq1.3})} and $\tilde u$ defined q.e. on $D$ by
\begin{equation}
\label{eq3.7}
\tilde u(x)=E_x\Big(e^{-\lambda\tau_D}g(X_{\tau_D})+\int^{\tau_D}_0e^{-\lambda t} f(X_t)\,dt\Big)
\end{equation}
is a quasi-continuous $m$-version of $u$.
\end{enumerate}
\end{proposition}
\begin{proof}
(i) Let $(\EE^{n\sigma},D(\EE^{n\sigma}))$ denote the form $\EE$ perturbed by the measure $n{\mathbf{1}}_{\partial D}\cdot\sigma$, that is
\[
\EE^{n\sigma}_{\lambda}(u,v)=\EE_{\lambda}(u,v)+n\int_{\partial D}uv\,d\sigma,\qquad u,v\in D(\EE^{n\sigma}):=H^1(D)\cap L^2(\bar D;{\mathbf{1}}_{\partial D}\cdot\sigma).
\]
By the classical trace theorem,  $D(\EE^{n\sigma})=H^1(D)$, so $u_n$ is a weak solution of (\ref{eq2.3}) if and only if $u_n\in D(\EE^{n\sigma})$ and
\begin{equation}
\label{eq2.6}
\EE^{n\sigma}_{\lambda}(u_n,v)=\int_Dfv\,dx+n\int_{\partial D}gv\,d\sigma,\quad v\in D(\EE^{n\sigma}).
\end{equation}
Therefore we have to show that there is a unique $u_n\in H^1(D)$ satisfying (\ref{eq2.6}). Suppose that  $u^1_n,u^2_n\in H^1(D)$ satisfy (\ref{eq2.6}) and let $u=u^1_n-u^2_n$. Then from (\ref{eq2.6}) with test function $v=u$ we get
$\EE^{n\sigma}_{\lambda}(u,u)=0$, hence that  $\EE_{\lambda}(u,u)=0$. Clearly, this
implies that $u=0$ $m$-a.e. To prove the existence and its representation,
it suffices to note that $\tilde u_n$  can be written in the form
\[
\tilde u_n=R^{nA}_{\lambda}f+nU^{\lambda}_{n,A}g,
\]
where
\[
R^{nA}_{\lambda}f(x)=E_x\int^{\infty}_0e^{-\lambda t-nA_t}f(X_t)\,dt,\qquad
U^{\lambda}_{n,A}g(x)=E_x\int^{\infty}_0e^{-\lambda t-nA_t}g(X_t)\,dA_t,
\]
and then use \cite[(6.1.5), (6.1.12)]{FOT}. Furthermore, $\tilde u_n$ is quasi-continuous because $R^{nA}_{\lambda}f$ is quasi-continuous by \cite[Lemma 5.1.5]{FOT} and $U^{\lambda}_{n,A}g$ is quasi-continuous by \cite[Lemma 6.1.3]{FOT}.
\\
(ii) With our convention,  $\tilde u$  can be equivalently written in the form
\[
\tilde u=H^{\lambda}_{\partial D}\tilde g+R^D_{\lambda}f,
\]
where
\[
H^{\lambda}_{\partial D}\tilde g(x)=E_xe^{-\lambda\tau_D}\tilde g(X_{\tau_D}),
\qquad R^D_{\lambda}f(x)=\int^{\infty}_0e^{-\lambda t}f(X^D_t)\,dt.
\]
Let $H^1_D=\{u\in H^1(D):\tilde u=0\mbox{ q.e. on }\partial D\}$. It is known (see \cite[Exercise 2.3.1]{FOT}) that $H^1_D=H^1_0(D)$. Furthermore, by \cite[Theorem 4.3.1]{FOT}, $H^{\lambda}_{\partial D}\tilde g$ is an $m$-version of  the orthogonal projection of $g$ on the orthogonal complement of the space $H^1_D$
in the Hilbert space $(H^1(D),\EE_{\lambda})$. Hence, for every $v\in H^1_0(D)$, $\EE_{\lambda}(H^{\lambda}_{\partial D}\tilde g,v)=0$.
Therefore, if $\tilde u$ is defined by (\ref{eq3.7}), then for  every $v\in H^1_0(D)$ we have
\[
\EE_{\lambda}(\tilde u,v)=\EE_{\lambda}(R^D_{\lambda}f,v)=\int_Dfv\,dx,
\]
the second equality being a consequence of \cite[Theorem 4.4.1]{FOT}, Furthermore, $\tilde u-g=\tilde u-(H^{\lambda}_{\partial D}\tilde g+g-H^{\lambda}_{\partial D}\tilde g)
=R^D_{\lambda}f-(g-H^{\lambda}_{\partial D}\tilde g)\in H^1_0(D)$ since $g-H^{\lambda}_{\partial D}\tilde g\in H^1_0(D)$ and $R^D_{\lambda}f\in H^1_0(D)$ by \cite[Theorem 4.4.1]{FOT} again. Therefore $\tilde u$ is a weak solution of (\ref{eq1.3}). Note that $\tilde u$ is quasi-continuous because $H^{\lambda}_{\partial D}\tilde g$ is quasi-continuous by \cite[Theorem 4.3.1]{FOT} and $R^D_{\lambda}f$ is quasi-continuous by \cite[Theorem 4.4.1]{FOT}.
\end{proof}

Note that since  $D$ is Lipschitz, there is the trace operator $\gamma:H^1(D)\rightarrow L^2(\partial D)$.
Therefore in Definition \ref{def3.1}(i) and Proposition \ref{prop2.4}(ii) one can assume that
$g\in H^1(D)$ and then replace $g$ by $\gamma(g)$ in (\ref{eq2.3}), (\ref{eq2.5}).

If $f\in L^p(D)$ with $p>d$,  then $R_{\lambda}|f|\in C(\bar D)$ by \cite[Theorem 2.1]{FT2}, and if $g\in \BB_b(\partial D)$, then   $nE_x\int^{\infty}_0e^{-n A_t}g(X_t)\,dA_t\le \|g\|_{\infty}E_x(1-e^{-nA_{\infty}})$, $x\in\bar D$. Therefore, under these assumptions on $f$ and $g$, the integrals on the right-hand side of (\ref{eq2.5}) are well defined for every $x\in\bar D$. Similarly,  the right-hand side of (\ref{eq3.7}) is well defined for every $x\in \bar D$.

The above remarks and Proposition \ref{prop2.4} justify the following definition of probabilistic solutions of (\ref{eq1.2}), (\ref{eq1.3}).

\begin{definition}
\label{def3.3}
Let $f\in L^p(D)$ with $p>d$ and $g\in\BB_b(\partial D)$.  The function $v_n:\bar D\rightarrow\BR$  defined by the right-hand side of \mbox{\rm(\ref{eq2.5})} is called the probabilistic solution of (\ref{eq1.2}).
The function  $v:\bar D\rightarrow\BR$ defined by the right-hand side of (\ref{eq3.7}) is  called the probabilistic solution of (\ref{eq1.3}).
\end{definition}

An equivalent definition of a probabilistic solution of (\ref{eq1.2}), resembling (\ref{eq2.3}), will be given in Proposition \ref{prop3.4} below.

For a deep study of connections between probabilistic solutions, weak solutions as well of other kind of solutions to the Dirichlet problem  with possibly irregular domain we refer the reader to \cite{K:JFA}. Here let us only note that if  $D$ is bounded and Lipschitz (as in the present paper), then it satisfies Poincare's cone condition. Therefore modifying slightly the proof of \cite[Proposition II.1.13]{B} (we use Aronson's estimates for the transition densities of $\BM$) one can show that   each point  $x\in\partial D$ is regular for $D^c$, i.e.
\begin{equation}
\label{eq3.5}
P_{x}(\tau_D=0)=1, \quad x\in\partial D.
\end{equation}
Using this, similarly to the proof of \cite[Proposition II.1.11]{B}, one can show that $H^{\lambda}_{\partial D}g\in C(\bar D)$ if $g\in C(\partial D)$. For an analytical
proof of this well known fact see, e.g., \cite{LSW}. Furthermore, it is known (see \cite[Section 9]{S} or \cite{K}) that if $f\in L^p(D)$ with $p>d$, then $R^D_{\lambda}f\in C(\bar D)$. Thus $v\in C(\bar D)$ when $f\in L^p(D)$ with $p>d$ and $g\in C(\partial D)$.
\begin{proposition}
\label{prop3.4}
Let $f\in L^p(D)$ with $p>d$ and $g\in\BB_b(\partial D)$. Then the probabilistic solution $v_n$ of \mbox{\rm(\ref{eq1.2})} is continuous. Moreover, $v_n\in C(\bar D)$ is the probabilistic solution if and only if it satisfies the equation
\begin{align}
\label{eq3.6}
v_n(x)&=R_{\lambda}(f\cdot m+n(g-v_n)\cdot\sigma)(x)\nonumber\\
&=E_x\int^{\infty}_0e^{-\lambda t}(f(X_t)\,dt+n(g-v_n)(X_t)\,dA_t),\quad x\in\bar D.
\end{align}
\end{proposition}
\begin{proof}
Define $u_n,\tilde u_n$  as in Proposition \ref{prop2.4} and set
\begin{align}
\label{eq4.1}
w_n(x)&=R_{\lambda}(f\cdot m+n(g-\tilde u_n)\cdot\sigma)(x)\nonumber\\
&=\int_{D}r_{\lambda}(x,y)f(y)\,dy
+n\int_{\partial D}r_{\lambda}(x,y)(g-\tilde u_n)(y)\,\sigma(dy),\quad x\in\bar D.
\end{align}
By the remarks following the proof of Proposition \ref{prop2.4}, $w_n(x)$ is well defined and finite for each $x\in\bar D$. Moreover, there is $C>0$ such that $|\tilde u_n|\le C$  q.e. Since $\sigma$ is smooth, $|\tilde u_n|\le C$ $\sigma$-a.e. on $\partial D$. From this and  \cite[Theorem 2.1]{FT2} it follows that in fact $w_n\in C(\bar D)$.
For every $v\in H^1(D)$ we have
\[
\EE_{\lambda}(w_n,v)=(f,v)+n\int_{\partial D}(g-\tilde u_n)v\,d\sigma
=(f,v)+n\int_{\partial D}(g-u_n)v\,d\sigma.
\]
By this and (\ref{eq2.3}), $\EE_{\lambda}(w_n,v)=\EE_{\lambda}(u_n,v)$, $v\in H^1(D)$, which implies that $w_n=u_n$ $m$-a.e., and hence $w_n=\tilde u_n$ q.e. on $\bar D$. From this and (\ref{eq4.1}) it follows that $w_n$ is a continuous solution of (\ref{eq3.6}).
It is the probabilistic solution of (\ref{eq1.2}). To see this, we first note that (\ref{eq3.6}), with $v_n$ replaced by $w_n$,  can be equivalently written as
\begin{equation}
\label{eq3.9}
w_n(x)=E_x\int^{\infty}_0(f(X^{\lambda}_t)\,dt+n(g-w_n)(X^{\lambda}_t)\,dA_t),
\quad x\in \bar D.
\end{equation}
Since the integrals $E_x\int^{\infty}_0|f(X^{\lambda}_t)|\,dt$, $E_x\int^{\infty}_0|g-w_n|(X^{\lambda}_t)\,dA_t$
exist and are finite for each $x\in \bar D$, in much the same way as
in \cite[Remark 3.3(ii)]{KR:NoD}  we show that there is a martingale additive
functional $M$ such that for each $x\in\bar D$ the pair $(Y^n,M)$, where $Y^n_t=w_n(X^{\lambda}_t)$, $t\ge0$, is a solution of the backward stochastic
differential equation
\begin{equation}
\label{eq3.8}
Y^n_t=\int^{\infty}_tf(X^{\lambda}_s)\,ds
+n\int^{\infty}_t(g(X^{\lambda}_s)-Y^n_s)\,dA_s
-\int^{\infty}_tdM_s, \quad t\ge0, \quad P_x\mbox{-a.s.}
\end{equation}
Integrating by parts, we get
\[
e^{-n A_T}Y^n_T-Y^n_0=-n\int^T_0e^{-n A_t}Y^n_t\,dA_t
+\int^T_0e^{-n A_t}\,dY^n_t,\quad T>0.
\]
Hence
\[
E_xY^n_0=E_xe^{-nA_T}Y^n_T+\int^T_0e^{-nA_t}(f(X^{\lambda}_t)\,dt
+ng(X^{\lambda}_t)\,dA_t).
\]
Letting $T\rightarrow\infty$ gives
\begin{align*}
w_n(x)=E_xY^n_0&=E_x\int^{\infty}_0e^{-nA_t}(f(X^{\lambda}_t)\,dt
+ng(X^{\lambda}_t)\,dA_t)\\
&=E_x\int^{\infty}_0e^{-\lambda t-nA_t}(f(X_t)\,dt+ng(X_t)\,dA_t)=v_n(x)
\end{align*}
for every $x\in \bar D$. This shows that $v_n$ is continuous and satisfies (\ref{eq3.6}), and moreover, any continuous solution of (\ref{eq3.9}) coincides with $v_n$.
\end{proof}

Note that (\ref{eq3.6}) is a very special case of equation with smooth measure data and (\ref{eq3.8}) is the corresponding backward stochastic differential equation (BSDE). More general, semilinear equations of the form (\ref{eq3.6}), (\ref{eq3.8}) are considered in  \cite{KR:PA}. Note also that  one can prove the existence of a quasi-continuous $v_n$ satisfying (\ref{eq3.6}) for q.e. $x\in\bar D$ by solving the corresponding BSDE, i.e. by probabilistic methods (we do not need to know in advance that there is a weak solution $u_n$). For a general result of this kind see \cite[Theorem 4.3]{KR:PA}.

\section{A convergence result}

Recall that $A$ is an additive functional (AF in abbreviation) of $\BM$ in the strict sense with Revuz measure $\sigma$.  We denote by $F_A$ the support of $A$, i.e.
\[
F_A=\{x\in\bar D:P_x(A_t>0\mbox{\rm\ for all }t>0)=1\}.
\]

\begin{lemma}
\label{lem4.1} $P_x(A_{t\wedge\tau_D}=0,\,t\ge0)=1$ and $P_x(A_{t+\tau_D}>0,\,t>0)=1$ for every $x\in\bar D$.
\end{lemma}
\begin{proof}
In view of (\ref{eq3.5}), the first part of the lemma is trivial for $x\in\partial D$. To show it for $x\in D$, we denote  by $F$ the quasi-support of $\sigma$. We may and will assume that $F\subset\partial D$ (see \cite[p. 190]{FOT}). Since $A$ is an AF in the strict sense,
by  \cite[Lemma 5.1.11]{FOT} we  have $P_x(A_t=(\fch_{F_A}\cdot A)_t,t>0)=1$ for every $x\in \bar D$, where $(\fch_{F_A}\cdot A)_t=\int^t_0\fch_{F_A}(X_s)\,dA_s$, $t\ge0$. By \cite[Theorem 5.1.5]{FOT}, $F_A=F$, so $P_x(A_t=(\fch_{F}\cdot A)_t,t>0)=1$ for every $x\in \bar D$. Since $F\subset\partial D$, it follows that for  $x\in D$, $A_t=0$ $P_x$-a.s. on $[0,\tau_D)$. Since $A$ is continuous, in fact $A_t=0$ $P_x$-a.s. on $[0,\tau_D]$ for $x\in D$, which proves the first part of the lemma.
Let $B$ be a standard Brownian motion appearing in (\ref{eq2.9}). We have  $P_{y}(\bar\tau_D=0)=1$ for $y\in\partial D$,  where $\bar\tau_D=\inf\{t>0:B_t\notin D\}$. From this, (\ref{eq2.9}) and the fact that the reflecting Brownian motion  is a diffusion with sample paths in $\bar D$ it follows that the support of the additive functional  appearing in (\ref{eq2.9}), which we denote for the moment  by $\bar A$, equals $\partial D$.
Let $\mbox{Cap}_L$ denote the capacity associated with $\EE$ and
Cap the capacity associated with $\BD$ (see \cite[Section 2.1]{FOT} for the definitions). Assumption (\ref{eq1.1}) implies
that $2\lambda^{-1}\mbox{Cap}\le\mbox{Cap}_L\le 2\lambda\mbox{Cap}$. Therefore $F$ is a quasi-support of $\sigma$ considered as a smooth measure with respect to $\mbox{Cap}_L$ if and only if it is a quasi-support of $\sigma$ considered as a smooth measure with respect to Cap. By what has already been proved and \cite[Theorem 5.1.5]{FOT},  $F=F_{\bar A}=\partial D$, so by \cite[Theorem 5.1.5]{FOT} again, $F_A=\partial D$. From this and the definition of $F_A$ we get the second part of the lemma.
\end{proof}


\begin{theorem}
\label{th4.2}
Assume that $f\in L^p(D)$ with $p>d$ and $g\in C(\partial D)$.
Then $v_n(x)\rightarrow v(x)$ for every $x\in\bar D$.
\end{theorem}
\begin{proof}
Recall that $v_n$ is defined by the right-hand side of (\ref{eq2.5}). First assume that $x\in D$. By Lemma \ref{lem4.1} and the dominated convergence theorem, for $x\in D$, we have
\begin{align}
\label{eq3.1}
E_x\int^{\infty}_0e^{-\lambda t-nA_t}f(X_t)\,dt
&=E_x\int^{\tau_D}_0e^{-\lambda t}f(X_t)\,dt
+E_x\int^{\infty}_{\tau_D}e^{-\lambda t-nA_t}f(X_t)\,dt\nonumber\\
&\quad\rightarrow E_x\int^{\tau_D}_0e^{-\lambda t}f(X_t)\,dt=R^D_{\lambda}f(x)
\end{align}
as $n\rightarrow\infty$. We are going to show that for every $x\in D$,
\begin{align}
\label{ew4.2}
nE_x\int^{\infty}_0e^{-\lambda t-nA_t}g(X_t)\,dA_t
&=nE_x\int^{\infty}_{\tau_D}e^{-\lambda t-nA_t}g(X_t)\,dA_t\nonumber\\
&\quad\rightarrow E_xe^{-\lambda\tau_D}g(X_{\tau_D})=H^{\lambda}_{\partial D}g(x)
\end{align}
as $n\rightarrow\infty$. We know that $(P_t)_{t>0}$ is a strongly Feller semigroup on $C(\bar D)$. Let
$(\hat L,D(\hat L))$ denote its generator. Since $D(\hat L)$ is dense in $C(\bar D)$, one can choose a sequence $\{g_k\}\subset D(\hat L)$ such that $\sup_{x\in\bar D}|g_k-g|\le k^{-1}$. By \cite[Theorem 3.6.5]{J3}, $g_k(X)$ is a semimartingale under $P_x$ for $x\in \bar D$. In fact,
\[
M^{g_k}_t:=g_k(X_t)-g_k(X_0)-\int^t_0(\hat Lg_k)(X_s)\,ds,\quad t\ge0,
\]
is a martingale under $P_x$ for $x\in \bar D$. Integrating by parts, for all $k\ge1$ and $t\ge0$ we obtain
\begin{align*}
&e^{-\lambda( t+\tau_D)-nA_{t+\tau_D}}g_k(X_t)
-e^{-\lambda\tau_D-nA_{\tau_D}}g_k(X_{\tau_D})\\
&\quad=-\int^{t+\tau_D}_{\tau_D}e^{-\lambda s-nA_s}g_k(X_s)\,d(\lambda s+nA_s)
+\int^{t+\tau_D}_{\tau_D}e^{-\lambda s-nA_s}\,dg_k(X_s) \\
&\qquad +\int^{t+\tau_D}_{\tau_D}e^{-\lambda s-nA_s}\,dM^{g_k}_s.
\end{align*}
Since $e^{-\lambda t-nA_t}\rightarrow0$ as $t\rightarrow\infty$ and $A_{\tau_D}=0$ $P_x$-a.s., we get
\begin{align*}
nE_x\int^{\infty}_{\tau_D}e^{-\lambda s-nA_s}g_k(X_s)\,dA_s &=E_xe^{-\lambda\tau_D}g_k(X_{\tau_D})
-\lambda E_x\int^{\infty}_{\tau_D}e^{-\lambda s-nA_s}g_k(X_s)\,ds\nonumber\\
&\quad+E_x\int^{\infty} _{\tau_D}e^{-\lambda s-nA_s} (\hat Lg_k)(X_s)\,ds.
\end{align*}
Since $g_k,\hat Lg_k\in C(\bar D)$,  
applying Lemma \ref{lem4.1} and the dominated convergence theorem shows that the   second and third term on the right-hand side of the above equality  converge to zero as $n\rightarrow\infty$.   This proves that
\begin{equation}
\label{eq3.2}
n E_x\int^{\infty}_0e^{-\lambda s-nA_s}g_k(X_s)\,dA_s
\rightarrow E_xe^{-\lambda\tau_D}g_k(X_{\tau_D}).
\end{equation}
Furthermore,
\[
n\int^{\infty}_{\tau_D}e^{-\lambda s-nA_s}\,dA_s\le ne^{-\lambda\tau_D}\int^{\infty}_0e^{-n A_s}\,dA_s= e^{-\lambda\tau_D}(1-e^{-nA_{\infty}}),
\]
so
\begin{equation}
\label{eq3.3}
n E_x\int^{\infty}_{\tau_D}e^{-\lambda s-nA_s}|g_k-g|(X_s)\,dA_s
\le k^{-1}E_xe^{-\lambda\tau_D}.
\end{equation}
Clearly, we also have
\begin{equation}
\label{eq3.4}
E_xe^{-\lambda\tau_D}|g_k-g|(X_{\tau_D})\le k^{-1}.
\end{equation}
From  (\ref{eq3.2})--(\ref{eq3.4}) we get (\ref{ew4.2}), which together with (\ref{eq3.1}) shows the desired convergence for $x\in D$.
Since $P_x(\tau_D=0)=1$ for $x\in\partial D$,
the above arguments also show that $v_n(x)\rightarrow E_xg(X_0)=g(x)=v(x)$ for $x\in\partial D$,  which completes the proof.
\end{proof}

\begin{remark}
Let $f\in L^2(D)$, $g\in C(\partial D)$ and $\tilde u_n,\tilde u$  be defined as in Proposition \ref{prop2.4}. Then $\tilde u_n\rightarrow u$ q.e. because the proof of Theorem \ref{th4.2} shows that then (\ref{eq3.1}) holds for q.e. $x\in D$ and  (\ref{ew4.2}) holds for every $x\in D$. In particular, if $f\in L^2(D)$ and $g\in H^1(D)\cap C(\partial D)$, then $\{u_n\}$ converges q.e. to the weak solution $u$ of (\ref{eq1.3}). If $f\in L^2(D),g\in H^1(D)$, then the convergence holds in $H^1(D)$ and hence a.e. For an analytical proof of this fact we refer the reader to \cite[Appendix I, Section 4.4]{G}.
\end{remark}

\end{document}